\documentclass[12pt]{article}
\usepackage{ct}

\usepackage{mathtools,tabularx,multirow,rotating,algorithm,algorithmic,placeins}

\specs{n}{vol (issue)}{year}{}

\dateline{Apr 4, 2023}{Jan 23, 2024}{TBD}



\title{Counting Conjugacy Classes of Elements of Finite Order in Exceptional Lie Groups}

\author[1]{Tamar Friedmann}
\author[2]{Qidong He}

\affil[1]{%
Department of Mathematics, Colby College, Waterville, ME, U.S.A.

\email{tfriedma@colby.edu}%
}

\affil[2]{%
Department of Mathematics, Rutgers University, Piscataway, NJ, U.S.A.

\email{qh97@math.rutgers.edu}%
}

\newcommand{\Z}{\mathbb Z}
\newcommand{\C}{\mathbb C}
\newcommand{\R}{\mathbb R}
\newcommand{\e}{\mathfrak e}
\newcommand{\f}{\mathfrak f}
\newcommand{\g}{\mathfrak g}
\newcommand{\h}{\mathfrak h}
\newcommand{\Mcal}{\mathcal M}
\newcommand{\Scal}{\mathcal S}
\newcommand{\Cl}{\operatorname{Cl}}
\newcommand{\diag}{\operatorname{diag}}
\newcommand{\Fix}{\operatorname{Fix}}
\newcommand{\M}{\operatorname{M}}
\newcommand{\row}{\operatorname{row}}
\newcommand{\SNF}{\operatorname{SNF}}
\newcommand{\SU}{\operatorname{SU}}

\newcolumntype{Y}{>{\centering\arraybackslash}X}
\newcolumntype{P}[1]{>{\centering\arraybackslash}p{#1}}

\begin{document}

\maketitle


\begin{abstract}
This paper continues the study of two numbers that are associated with Lie groups. 
The first number is $N(G,m)$, the number of conjugacy classes of elements in $G$ whose order divides $m$. 
The second number is $N(G,m,s)$, the number of conjugacy classes of elements in $G$ whose order divides $m$ and which have $s$ distinct eigenvalues, where we view $G$ as a matrix group in its smallest-degree faithful representation. 
We describe systematic algorithms for computing both numbers for $G$ a connected and simply-connected exceptional Lie group. 
We also provide explicit results for all of $N(G,m)$, $N(G_2,m,s)$, and $N(F_4,m,s)$. 
The numbers $N(G,m,s)$ were previously known only for the classical Lie groups; our results for $N(G,m)$ agree with those already in the literature but are obtained differently. 
\end{abstract}



\section{Introduction}

Let $G$ be a complex, simply-connected Lie group, viewed as a matrix group via its standard representation (i.e., its smallest-degree, faithful representation), and let $m$ and $s$ be positive integers. We study  the number of conjugacy classes of elements in $G$ whose order divides $m$, as well as  the number of such classes whose elements have $s$ distinct eigenvalues. That is, we define
\begin{align*} E(G,m)&=\{ x\in G \; |\;  x^m=1\}, \\
\label{egms1} E(G,m,s)&=  \{ x\in E(G,m)\;  | \; x \mbox{ has } s \mbox{ distinct eigenvalues} \} ,
\end{align*}
and study 
\begin{align*} N(G,m)&= \mbox{number of conjugacy classes of } G \mbox{ in } E(G,m), \\
N(G,m,s)&= \mbox{number of conjugacy classes of } G \mbox{ in } E(G,m,s).
\end{align*}

The study of the number of conjugacy classes of elements of finite order in Lie groups has an interesting history that combines mathematical and physical approaches and applications. 

The story begins in the 1980's with a pair of papers by Djokovic \cite{Djo80, Djo85}, where nice formulas for $N(G,m)$ were obtained for any connected semisimple complex Lie group $G$ that is simply connected or adjoint, using a generating function approach. In \cite{pianzola1, pianzola2}, the case of certain prime power orders was computed; and in \cite{lossers}, $N(\SU(n),m)$ was obtained.
This topic was revived in \cite{Fried13m, Fried13p}, where $N(G,m)$ was obtained for unitary, orthogonal, and symplectic Lie groups using simple combinatorial methods that apply to groups that are not necessarily connected, simply connected, or adjoint. Also in \cite{Fried13m, Fried13p}, $E(G,m,s)$ and $N(G,m,s)$ were introduced for the first time, motivated by an explicitly enumerative problem in string theory.\footnote{In \cite{Fried13m, Fried13p}, the number $s$ denotes the number of distinct {\it pairs} of eigenvalues.} Formulas for $N(G,m,s)$ were obtained there for the unitary, orthogonal, and symplectic Lie groups.\footnote{Other aspects of elements of finite order in Lie groups have been
studied. See for example \cite{other3, other1, other2, other4, other5}.}

In this paper, we study the numbers  $N(G,m)$ and $N(G,m,s)$ for $G$ an exceptional  Lie group. For $N(G,m)$, our results coincide with those in \cite{Djo85} (with the exception of a misprint that we discovered therein; see Remark \ref{rem:misprint}) and are obtained in a different way. For $N(G,m,s)$, our results are new. 

In Section \ref{Worbits} we recall a few necessary ingredients: the Chevalley basis, Chevalley group, and the Weyl group; we introduce notation, and present our main counting formula we rely on throughout the paper (Theorem \ref{thm:burnsidereformulated}). Said formula is a modification of Burnside's Lemma. In Section \ref{ngm} we compute $N(G,m)$, first for $G_2$ (Section \ref{subsec:ng2m}) and then for all the other four exceptional groups (Section \ref{subsec:ngm}).
Section \ref{ngms} is devoted to $N(G,m,s)$. In Section 
\ref{subsec:ng2ms} we present our methods using $G=G_2$ as an example and obtain our counting formula for this case (Corollary \ref{cor:ng2ms}). Among other things, our methods involve imposing and exploiting a partial order on a certain set of submatrices.  In Section \ref{subsec:ngms}  we generalize to the other four exceptional groups. We provide a formula (Corollary \ref{cor:ngms}) 
that applies to all of the exceptional groups, and obtain explicit results for the smallest two, $G=G_2, F_4$. 
In an appendix, we include three pseudo-codes which summarize our methods.

\section{$W$-orbits and conjugacy classes}\label{Worbits}

Let $\g$ be a simple Lie algebra of rank $\ell$ over $\C$, let $\h$ be a Cartan subalgebra of $\g$, let $\Phi$ be its root system, and let $\Pi$ be a fundamental system of the roots. 
Then, $\g$ has a basis $\{ h_r, r\in \Pi \; ; e_r, r\in \Phi \}$, called a Chevalley basis, with respect to which the multiplication constants of $\g$ are all integers, and particularly $[e_r, e_{-r}]=h_r$; the set $\{ h_r, r\in\Pi\}$ forms a basis for $\h$. For $\zeta \in \C$, let 
\begin{equation}\label{xr} x_r(\zeta)=\exp (\zeta e_r ),\end{equation} 
where $e_r$ is viewed in the smallest faithful representation of $\g$. Then the Chevalley group $G=\left \langle x_r(\zeta)\mid  r\in \Phi, \zeta \in \C \right \rangle$ is a complex simply-connected Lie group  in its smallest faithful representation. Now, define
\begin{equation}\label{nrzetahrzeta} n_r(\zeta)\coloneqq x_r(\zeta)x_{-r}(-{\zeta}^{-1})x_r(\zeta)\hskip 1cm \mbox{and} \hskip 1cm h_r(\zeta)\coloneqq n_r(\zeta)n_r(-1) . \end{equation}
Let $T$ be a maximal torus of $G$ defined as a maximal connected diagonalizable subgroup of $G$.
Then, we have that
\begin{enumerate}
\item The elements $h_r(\zeta)$ of $G$ generate a subgroup isomorphic to $T$. 
\item The elements 
\begin{equation}\label{nr}n_r\coloneqq n_r(1)=x_r(1) x_{-r}(-1)x_r(1)\end{equation}
operate on $T$ by conjugation in the same way as the reflections of the Weyl group $W$ act on $T$. 
\end{enumerate}
See \cite{Car72, Ste67} for details.

In \cite{How01}, the $x_{\pm r}(\zeta)$ are provided in the smallest-degree faithful representation for each of the five exceptional Lie algebras $\g_2, \f_4, \e_6, \e_7, \e_8$. 
 The representations given there allow us to obtain explicitly the generators $n_r$ of $W$, the basis of the Cartan subalgebra $\h$ via $h_r=[e_r, e_{-r}]$, and the torus $T = \exp \h$ (the exponential map is surjective on tori). The torus is then
\begin{equation} T=\left\{\prod_{r\in\Pi}\exp(2\pi ik_r[e_r , e_{-r}])\mid k_r\in\R\right\}= \left \{ \prod_{r\in\Pi} h_r(e^{2\pi i k_r}) \mid k_r\in \R \right \},
\end{equation}
where $e_r$ are taken to be in the representation given in \cite{How01}. 
Let $(\delta_1,\dots,\delta_\ell)$ be the fundamental roots of $\g$, so that 
\begin{equation}\label{generaltypicaltorus}t(\vec{k})=t(k_{\delta_1},\dots,k_{\delta_\ell})=\Pi _{i=1}^\ell \exp(2\pi ik_{\delta_i}[e_{\delta_i}, e_{-\delta_i}])= \Pi _{i=1}^\ell h_{\delta _i}(e^{2\pi i k_{\delta_i}})
\end{equation}
denotes a typical element of $T$. Then, $t(\vec{k})$ has order dividing $m$ iff $\vec{k}\in \left (\frac{1}{m}\Z \right )^\ell$. Since $G$ is simply connected, two $\ell$-tuples 
$\vec{k},\vec{k}'\in \left (\frac{1}{m}\Z \right )^\ell$ represent the same element of the maximal torus iff $\vec{k}-\vec{k}'\in\Z^\ell$, so we will henceforth restrict ourselves to  $\vec{k}\in\left(\frac{1}{m}\Z/\Z\right)^\ell$.

Let $\Cl(W)$ be the collection of conjugacy classes in $W$, and let $w_c $ be a representative of a conjugacy class $c\in \Cl(W)$. Also, let $\left\lvert\Fix(w_c)\right\rvert$ be the number of fixed points of the set 
\[ T(G,m):=\left\{t(\vec{k}):\vec{k}\in\left(\frac{1}{m}\Z/\Z\right)^\ell \right \} \] 
under the action of $w_c$ and let $\left\lvert\Fix_s(w_c)\right\rvert$ be the number of fixed points of the set 
\[ T_s(G,m):= \left\{t(\vec{k}):\vec{k}\in\left(\frac{1}{m}\Z/\Z\right)^\ell,\ t(\vec{k})\text{ has $s$ distinct eigenvalues}\right\}\] under the same action. 

\begin{theorem} \label{thm:burnsidereformulated}
We have 
\begin{align} \label{ngmformula}
N(G,m)=N(T,m)&=\frac{1}{\lvert W\rvert}\sum_{c\in\Cl(W)}\left\lvert c\right\rvert\left\lvert\Fix(w_c)\right\rvert 
\end{align}
and
\begin{align}\label{ngmsformula}
N(G,m,s)=N(T,m,s)&=\frac{1}{\lvert W\rvert}\sum_{c\in\Cl(W)}\left\lvert c\right\rvert\left\lvert\Fix_s(w_c)\right\rvert .
\end{align}

\end{theorem}
\begin{proof}

We begin with the first equalities of Equations \eqref{ngmformula} and \eqref{ngmsformula}. 
Let $x\in G$ be an element of finite order $m$. 
Let $x=su$ be its multiplicative Jordan decomposition, with $s$ semisimple and $u$ unipotent, and $su=us$. Then $x^m=s^mu^m$ is the multiplicative Jordan decomposition of $x^m=1$. It follows that $u^m=1$ which implies $u=1$ so that $x$ is semisimple. So $x$ lies in a maximal torus of $G$ and is therefore conjugate to an element of $T$.
Furthermore, the order of any element in $G$ is invariant under conjugation. So we have $N(G,m)=N(T,m)$. Since the eigenvalues of any element in $G$ are also invariant under conjugation, we have $N(G,m,s)=N(T,m,s)$.
We have reduced the counting of conjugacy classes in $G$ to that of conjugacy classes in $T$.  

We now prove the second equalities of Equations \eqref{ngmformula} and \eqref{ngmsformula}. As noted also in \cite{Djo80}, two elements $t,s\in T$ are conjugate in $G$ iff there exists $w\in W$ satisfying $s=w\cdot t$, where $W$ acts on $T$ by conjugation (see for example \cite[Chap. 11]{Hall04}, \cite[Chap. 4]{Knapp96}). Hence, $N(T,m)$ is the number of orbits of elements in $T(G,m)$ under the action of $W$. By Burnside's Lemma, 
 \[ N(T,m)=\frac{1}{\lvert W\rvert}\sum_{w\in W}\left\lvert\Fix(w)\right\rvert,\] where $\Fix(w)=\{t\in T(G,m):w\cdot t=t \}$ is the set of fixed points of $T(G,m)$ under the action of $g$. Now, if $u,v \in W$ are conjugate, there is a canonical bijection between the fixed points of $T(G,m)$ under $u$ and under $v$, so $\lvert \Fix(u)\rvert =\lvert\Fix(v)\rvert$. Therefore, we can simplify the above sum over the Weyl group to the sum over its conjugacy classes given in the theorem. An analogous argument holds for $N(T,m,s)$. 
\end{proof}
The utility of Theorem \ref{thm:burnsidereformulated} stems from the fact that while the Weyl group of an exceptional Lie group may be forbiddingly large, the number of conjugacy classes inside the Weyl group is generally small. Hence, provided that we know the size and a representative of each conjugacy class in the Weyl group, the sums in Theorem \ref{thm:burnsidereformulated} will be much easier to evaluate than the ones given by Burnside's Lemma. Fortunately, this information has been completely determined in \cite{Car72} and translated into an accessible form in GAP 3 \cite{gap3, chevie}.

\begin{remark} Theorem  \ref{thm:burnsidereformulated} as well as the rest of the methods described in this paper apply to the classical Lie groups as well, and may be used to reproduce the results of \cite{Fried13m, Fried13p} for the simply-connected cases there. 
\end{remark}

\section{$N(G,m)$}\label{ngm}
In order to obtain $N(G,m)$, we need to compute $\lvert \Fix(w_c)\rvert$ for $c\in \Cl(W)$ and then use Theorem \ref{thm:burnsidereformulated}. We carry out the computation for $G=G_2$ in Section \ref{subsec:ng2m} and then generalize to any exceptional Lie group in Section \ref{subsec:ngm}. 
\subsection{Computation of $N(G_2,m)$.} \label{subsec:ng2m}
Let the two fundamental roots of $G_2$ be denoted $\delta _1$ and $\delta _2$. Then, \cite{How01} provides the following $7\times 7$ matrices
\begin{align*}
 \psi (e_{\delta_1})=& E_{12}+2E_{34}+E_{45}+E_{67} ,\\ 
\psi (e_{-\delta_1})=& E_{21}+E_{43}+2E_{54}+E_{76} ,\\
\psi (e_{\delta_2})=& E_{23}+E_{56},\\
\psi (e_{-\delta_2})=& E_{32}+E_{65},
\end{align*}
where $E_{ij}$ is the matrix with the entry 1 in the $ij$th position and 0 elsewhere. 
From these we obtain the following basis elements for the Cartan subalgebra
\begin{align*}
h_{e_{\delta_1}}&=[\psi (e_{\delta_1}),\psi (e_{-\delta_1})]=\diag(1,-1,2,0,-2,1,-1) ,\\
h_{e_{\delta_2}}&=[\psi (e_{\delta_2}),\psi (e_{-\delta_2})]=\diag(0,1,-1,0,1,-1,0), \\
\end{align*}
which give, via Equation \eqref{generaltypicaltorus},
\footnotesize
\begin{equation}\label{typicaltorusg2}
\small
t(k_{\delta_1},k_{\delta_2})=\diag(e^{2\pi i k_{\delta_1}},e^{-2\pi i(k_{\delta_1}-k_{\delta_2})},e^{2\pi i(2k_{\delta_1}-k_{\delta_2})},1,e^{-2\pi i(2k_{\delta_1}-k_{\delta_2})}, e^{2\pi i(k_{\delta_1}-k_{\delta_2})},e^{-2\pi ik_{\delta_1}})
\end{equation}
\normalsize
for a typical element of the torus. Furthermore, using Equations \eqref{xr} and \eqref{nr}, we obtain

\begin{equation} \label{ndeltag2}
n_{\delta_1}=\begin{bmatrix}&1&&&&&\\-1&&&&&&\\&&&&1&&\\&&&-1&&&\\&&1&&&&\\&&&&&&1\\&&&&&-1&\end{bmatrix}, \; \;
n_{\delta_2}=\begin{bmatrix}1&&&&&&\\&&1&&&&\\&-1&&&&&\\&&&1&&&\\&&&&&1&\\&&&&-1&&\\&&&&&&1\end{bmatrix}
\end{equation}
for the generators of $W$. The size and a representative of each conjugacy class in $W$ are shown in Table \ref{tab:weylgroupg2}; as we have mentioned earlier, they are based on the results 
of \cite{Car72} and can be accessed in GAP 3 \cite{gap3, chevie}.
\begin{table}
\centering
\begin{tabular}{|c|c|}
\hline
$\lvert c\rvert$ & Representative $w_c\in c$ \\ \hline
$1$ & $I$ \\ \hline
$3$ & $n_{\delta_1}$ \\ \hline
$3$ & $n_{\delta_2}$ \\ \hline
$2$ & $n_{\delta_2}n_{\delta_1}$ \\ \hline
$2$ & $\left(n_{\delta_2}n_{\delta_1}\right)^2$ \\ \hline
$1$ & $\left(n_{\delta_2}n_{\delta_1}\right)^3$ \\ \hline
\end{tabular}
\caption{The size and a representative of each conjugacy class $c$ in the Weyl group of $G_2$.}
\label{tab:weylgroupg2}
\end{table}

As can be seen in \eqref{ndeltag2}, the elements of the Weyl group are signed permutation matrices. They act on $t(k_{\delta_1},k_{\delta_2})\in T$ by conjugation, resulting in a signed permutation of the exponents of the first three diagonal entries of $t(k_{\delta_1},k_{\delta_2})$, with the last three entries permuted accordingly. Since the exponents of the first three diagonal entries satisfy the relation
$$ 2k_{\delta_1}-k_{\delta_2}=k_{\delta_1}-[-(k_{\delta_1}-k_{\delta_2})],$$
it follows that $t(k_{\delta_1},k_{\delta_2})$ is fixed by an element $w\in W$ iff $w\cdot t(k_{\delta_1},k_{\delta_2})$ and $t(k_{\delta_1},k_{\delta_2})$ agree in their first two diagonal entries. 

It remains to determine $\lvert\Fix(w_c)\rvert$ for each conjugacy class representative $w_c$.

\begin{example}[Computing $\lvert\Fix(\left(n_{\delta_2}n_{\delta_1}\right)^2\rvert$] \label{ex:g2fixedpoint}
Let 
\[w_c=\left(n_{\delta_2}n_{\delta_1}\right)^2=\begin{bmatrix}&&&&1&&\\&&&&&&1\\&1&&&&&\\&&&1&&&\\&&&&&-1&\\-1&&&&&&\\&&1&&&&\end{bmatrix} \] 
and let $t(k_{\delta_1},k_{\delta_2})$ be as given in \eqref{typicaltorusg2}. The action of $w_c$ on $t(k_{\delta_1},k_{\delta_2})$ is given by

\begin{equation*}
\begin{split}
w_ct(&k_{\delta_1},k_{\delta_2})w_c^{-1} \\
&=\diag(e^{-2\pi i(2k_{\delta_1}-k_{\delta_2})}, e^{-2\pi ik_{\delta_1}}, e^{-2\pi i(k_{\delta_1}-k_{\delta_2})}, 1,e^{2\pi i(k_{\delta_1}-k_{\delta_2})},e^{2\pi ik_{\delta_1}}, e^{2\pi i(2k_{\delta_1}-k_{\delta_2})}).
\end{split}
\end{equation*}
Directing our attention to the first two diagonal entries, we conclude that $\lvert\Fix(w_c)\rvert$ is the number of solutions to the equation
\begin{equation}\label{coefficient:eg}\begin{bmatrix} 3 & -1 \\ 0 & 1 \end{bmatrix}\begin{bmatrix} k_{\delta_1} \\ k_{\delta_2} \end{bmatrix}=\vec{0}.\end{equation}
over $\frac{1}{m}\Z/\Z$. To count the solutions, we use the Smith normal form of an integer matrix.

\begin{theorem}[Smith Normal Form]\label{thm:snf}
If $A\in\M_{n\times \ell}(\Z)$, then there exist unimodular matrices $P\in\M_n(\Z)$ and $Q\in\M_\ell (\Z)$ such that 
\begin{equation}\label{matrix:snf}PAQ=\begin{bmatrix}d_1&&&&\\&\ddots&&&\\&&d_r&&\\&&&0&\\&&&&\ddots\end{bmatrix}=:\SNF(A),\end{equation}
where each $d_i\ne 0$, $d_i\mid d_j$ whenever $i\le j$, and the $d_i$'s are unique up to a sign. The matrix $\SNF(A)$ is called the \emph{Smith normal form} of $A$; the integers $d_1,\dots,d_r$ are called the \emph{elementary divisors} of $A$.
\end{theorem} 
Notice that the unimodularity of $P$ and $Q$ guarantees the existence of a bijection between $\ker A$ and $\ker\SNF(A)$ for all $m$ (recall that the kernels are considered over $\frac{1}{m}\Z/\Z$).

\begin{proposition} \label{prop:countkernel}
Let $A\in\M_{n\times \ell}(\Z)$. For any positive integer $m$, the size of the kernel of $A$ over $\frac{1}{m}\Z/\Z$ is given by \[\lvert\ker A\rvert=m^{\ell-r}\cdot\prod_{i=1}^r\gcd(d_i,m),\] where $d_1,\dots,d_r$ are the elementary divisors of $A$.
\end{proposition}
\begin{proof}
Consider the alternative equation $\SNF(A)\vec{k}=0$. The number of solutions to $0k_{\delta_i}=0$ is $m$ for $i=r+1, \ldots , \ell$; the number of solutions to $d_ik_{\delta_i}=0$ is $\gcd(d_i,m)$ for $i=1, \ldots , r$.
\end{proof}
\noindent The Smith normal form of the coefficient matrix in Equation \eqref{coefficient:eg} is \[\SNF\left(\begin{bmatrix} 3 & -1 \\ 0 & 1 \end{bmatrix}\right)=\begin{bmatrix} 1 & 0 \\ 0 & 3 \end{bmatrix}.\] It follows from Proposition \ref{prop:countkernel} that \[\lvert\Fix(\left(n_{\delta_2}n_{\delta_1}\right)^2)\rvert=\gcd(1,m)\gcd(3,m)=\begin{cases}3 & m\equiv 0\bmod 3\\1 & m\not\equiv 0\bmod 3\end{cases}.\]
\end{example}

\begin{table}
\centering
\begin{tabular}{|c|c|c|c|}
\hline
$w_c\in c$ & Associated System & SNF & $\lvert\Fix(w_c)\rvert$ \\ \hline
$I$ & $\begin{bmatrix}0 & 0 \\ 0 & 0\end{bmatrix}\vec{k}=\vec{0}$ & $\begin{bmatrix}0 & 0 \\ 0 & 0\end{bmatrix}$ & $m^2$ \\ \hline

$n_{\delta_1}(1)$ & $\begin{bmatrix}2 & -1 \\ 2 & -1\end{bmatrix}\vec{k}=\vec{0}$ & $\begin{bmatrix}1 & 0 \\ 0 & 0\end{bmatrix}$ & $m$ \\ \hline

$n_{\delta_2}(1)$ & $\begin{bmatrix}0 & 0 \\ 3 & -2\end{bmatrix}\vec{k}=\vec{0}$ & $\begin{bmatrix}1 & 0 \\ 0 & 0\end{bmatrix}$ & $m$ \\ \hline

$n_{\delta_2}(1)n_{\delta_1}(1)$ & $\begin{bmatrix}2 & -1 \\ 1 & 0\end{bmatrix}\vec{k}=\vec{0}$ & $\begin{bmatrix}1 & 0 \\ 0 & 1\end{bmatrix}$ & $1$ \\ \hline

$\left(n_{\delta_2}(1)n_{\delta_1}(1)\right)^2$ & $\begin{bmatrix}3 & -1 \\ 0 & 1\end{bmatrix}\vec{k}=\vec{0}$ & $\begin{bmatrix}1 & 0 \\ 0 & 3\end{bmatrix}$ & $\begin{cases}3 & m\equiv 0\bmod 3\\1 & m\not\equiv 0\bmod 3\end{cases}$ \\ \hline

$\left(n_{\delta_2}(1)n_{\delta_1}(1)\right)^3$ & $\begin{bmatrix}2 & 0 \\ 2 & -2\end{bmatrix}\vec{k}=\vec{0}$ & $\begin{bmatrix}2 & 0 \\ 0 & 2\end{bmatrix}$ & $\begin{cases}4 & m\equiv 0\bmod 2\\1 & m\not\equiv 0\bmod 2\end{cases}$ \\ \hline
\end{tabular}
\caption{The number of fixed points in $T(G_2,m)$ under the action of each conjugacy class $c\in\Cl(W)$ as represented by $w_c\in c$.}
\label{tab:fixedpointg2}
\end{table}

Repeating this computation for the other conjugacy classes in $W$ yields Table \ref{tab:fixedpointg2}. Finally, using Theorem \ref{thm:burnsidereformulated}, we obtain the following result.

\begin{theorem}
For any positive integer $m$, the number $N(G_2,m)$ is given by Table \ref{ng2m}.
\begin{table}
\centering
\def\arraystretch{1.6}
\begin{tabular}{|c|c|}
\hline
$m\bmod 6$ & $N(G_2,m)$ \\ \hline
$0$ & $(m^2+6m+12)/12$ \\ \hline
$\pm 1$ & $(m^2+6m+5)/12$ \\ \hline
$\pm 2$ & $(m^2+6m+8)/12$ \\ \hline
$3$ & $(m^2+6m+9)/12$ \\ \hline
\end{tabular}
\caption{Results for $N(G_2,m)$}
\label{ng2m}
\end{table}
\end{theorem}

\subsection{Computation of $N(G,m)$ for all exceptional Lie groups.}\label{subsec:ngm}
We compute $N(G,m)$ using the same method presented in Section \ref{subsec:ng2m}. Using the representations provided in \cite{How01}, we find that, in a suitable basis, the typical element of the torus, i.e., $t(\vec{k})$ of Equation \eqref{generaltypicaltorus}, takes the form (we include the torus of $G_2$ here again for reference)

\begin{enumerate}

 \newcommand\bigzero{\makebox(0,0){\text{\huge0}}}
\item 
$\displaystyle \diag(\exp(2\pi iP_1),\dots,\exp(2\pi iP_u),1,\exp(-2\pi iP_{u}),\dots,\exp(-2\pi iP_{1})),$

where $u=3$, if $G=G_2$;

\item 
$\displaystyle \diag(\exp(2\pi iP_1),\dots,\exp(2\pi iP_u),1,\exp(-2\pi iP_{u}),\dots,\exp(-2\pi iP_{1})),$

where $u=12$, if $G=F_4$;

\item 
$\displaystyle \diag(\exp(2\pi iP_1),\dots,\exp(2\pi iP_{u})),$

where $u=27$, if $G=E_6$;
\item 
$\displaystyle \diag(\exp(2\pi iP_1),\dots,\exp(2\pi iP_{u}),\exp(-2\pi iP_{u}),\dots,\exp(-2\pi iP_{1})),$

where $u=28$, if $G=E_7$; or 
\item 
$\displaystyle \diag(\exp(2\pi iP_1),\dots,\exp(2\pi iP_{u}),\underbrace{1,\dots,1}_8,\exp(-2\pi iP_{u}),\dots,\exp(-2\pi iP_{1})),$

where $u=120$, if $G=E_8$. 
\end{enumerate}
Here, for each $G$, the $P_1,\dots,P_u$ are distinct, nonzero, integer combinations of $k_{\delta_1},\dots,k_{\delta_\ell}$. For example, for $G_2$, we see from Equation \eqref{typicaltorusg2} that 
\begin{align}\label{p1p2p3}
P_1&=k_{\delta_1} ,\nonumber \\ P_2&=-(k_{\delta_1}-k_{\delta_2}),\\ P_3&= 2k_{\delta_1}-k_{\delta_2}.\nonumber
\end{align}
As with $G=G_2$, we use the results of \cite{Car72}.

For $G_2$, we saw that fixing two of the three diagonal entries of an element in $t(\vec{k})$ was sufficient to fix the element. It turns out that an analogue of this statement still applies in the general case.

\begin{proposition} \label{prop:wscorrespondence}
Let $w\in W$, where $W$ is the Weyl group of $G_2$, $F_4$, $E_6$, $E_7$, or $E_8$. Acting on $t(\vec{k})$ by $w$ results in a signed permutation of the exponents in its first $u$ diagonal entries. Further, there exist $\ell$ indices $\{i_1,\dots,i_\ell\}\subseteq\{1,\dots,u\}$ such that, for any $w\in W$, $t(\vec{k})$ is fixed by $w$ iff $w\cdot t(\vec{k})$ and $t(\vec{k})$ agree in their $i_j$th diagonal entries for each $j=1,\dots,\ell$.
\end{proposition}

\begin{proof}
Computationally verified. For the second claim, one simply needs to find $P_{i_1},\dots,P_{i_\ell}$ that form a $\Z$-spanning set of $\Z[k_{\delta_1},\dots,k_{\delta_\ell}]$.
\end{proof}

The procedure for computing each $\lvert\Fix(w_c)\rvert$ is then completely analogous to what we have done in Example \ref{ex:g2fixedpoint}; see Algorithm \ref{alg:fixedpoint} in the appendix.

Finally, applying Theorem \ref{thm:burnsidereformulated} yields the following results. 

\begin{theorem}
For any positive integer $m$, the numbers $N(F_4,m)$, $N(E_6,m)$, $N(E_7,m)$, and $N(E_8,m)$ are given respectively by Tables \ref{nf4m}, \ref{ne6m}, \ref{ne7m}, and \ref{ne8m}.
\begin{table}\centering
\def\arraystretch{1.6}
\begin{tabular}{|c|c|}
\hline
$m\bmod 12$ & $N(F_4,m)$ \\ \hline
$0$ & $(m^4+24m^3+208m^2+768m+1152)/1152$ \\ \hline
$\pm1,\pm5$ & $(m^4+24m^3+190m^2+552m+385)/1152$ \\ \hline
$\pm2$ & $(m^4+24m^3+208m^2+768m+880)/1152$ \\ \hline
$\pm3$ & $(m^4+24m^3+190m^2+552m+513)/1152$ \\ \hline
$\pm4$ & $(m^4+24m^3+208m^2+768m+1024)/1152$ \\ \hline
$6$ & $(m^4+24m^3+208m^2+768m+1008)/1152$ \\ \hline
\end{tabular}
\caption{Results for $N(F_4,m)$}
\label{nf4m}
\end{table}

\begin{table}\centering
\def\arraystretch{1.6}
\begin{tabular}{|c|c|}
\hline
$m\bmod 6$ & $N(E_6,m)$ \\ \hline
$0$ & $(m^6+36m^5+510m^4+3600m^3+14184m^2+35424m+51840)/51840$ \\ \hline
$\pm1$ & $(m^6+36m^5+510m^4+3600m^3+13089m^2+22284m+12320)/51840$ \\ \hline
$\pm2$ & $(m^6+36m^5+510m^4+3600m^3+13224m^2+23904m+16640)/51840$ \\ \hline
$3$ & $(m^6+36m^5+510m^4+3600m^3+14049m^2+33804m+38880)/51840$ \\ \hline
\end{tabular}
\caption{Results for $N(E_6,m)$}
\label{ne6m}
\end{table}

\begin{table}\centering
\def\arraystretch{1.6}
\begin{tabularx}{\textwidth}{|P{0.12\textwidth}|Y|}
\hline
$m\bmod 12$ & $N(E_7,m)$ \\ \hline
$0$ & $(m^7+63m^6+1617m^5+22050m^4+175224m^3+830592m^2+2239488m+2903040)/2903040$ \\ \hline
$\pm1,\pm5$ & $(m^7+63m^6+1617m^5+21735m^4+162939m^3+663957m^2+1286963m+765765)/2903040$ \\ \hline
$\pm2$ & $(m^7+63m^6+1617m^5+22050m^4+175224m^3+830592m^2+2176208m+2126880)/2903040$ \\ \hline
$\pm3$ & $(m^7+63m^6+1617m^5+21735m^4+162939m^3+663957m^2+1304883m+927045)/2903040$ \\ \hline
$\pm4$ & $(m^7+63m^6+1617m^5+22050m^4+175224m^3+830592m^2+2221568m+2580480)/2903040$ \\ \hline
$6$ & $(m^7+63m^6+1617m^5+22050m^4+175224m^3+830592m^2+2194128m+2449440)/2903040$ \\ \hline
\end{tabularx}
\caption{Results for $N(E_7,m)$}
\label{ne7m}
\end{table}

\begin{table}\centering
\def\arraystretch{1.6}
\begin{tabularx}{\textwidth}{|P{0.23\textwidth}|Y|}
\hline
$m\bmod 60$ & $N(E_8,m)$ \\ \hline
$0$ & $(m^8+120m^7+6020m^6+163800m^5+2626008m^4+25260480m^3+142577280m^2+445824000m+696729600)/696729600$ \\ \hline
$\begin{aligned}
\pm1,\pm7&,\pm11,\pm13,\\ \pm17,\pm1&9,\pm23,\pm29
\end{aligned}$ & $(m^8+120m^7+6020m^6+163800m^5+2616558m^4+24693480m^3+130085780m^2+323507400m+215656441)/696729600$ \\ \hline
$\pm2,\pm14,\pm22,\pm26$ & $(m^8+120m^7+6020m^6+163800m^5+2626008m^4+25260480m^3+141860480m^2+418876800m+435250816)/696729600$ \\ \hline
$\pm3,\pm9,\pm21,\pm27$ & $(m^8+120m^7+6020m^6+163800m^5+2616558m^4+24693480m^3+130802580m^2+345011400m+348264441)/696729600$ \\ \hline
$\pm4,\pm8,\pm16,\pm28$ & $(m^8+120m^7+6020m^6+163800m^5+2626008m^4+25260480m^3+141860480m^2+424320000m+516898816)/696729600$ \\ \hline
$\pm5,\pm25$ & $(m^8+120m^7+6020m^6+163800m^5+2616558m^4+24693480m^3+130085780m^2+323507400m+243525625)/696729600$ \\ \hline
$\pm6,\pm18$ & $(m^8+120m^7+6020m^6+163800m^5+2626008m^4+25260480m^3+142577280m^2+440380800m+587212416)/696729600$ \\ \hline
$\pm10$ & $(m^8+120m^7+6020m^6+163800m^5+2626008m^4+25260480m^3+141860480m^2+418876800m+463120000)/696729600$ \\ \hline
$\pm12,\pm24$ & $(m^8+120m^7+6020m^6+163800m^5+2626008m^4+25260480m^3+142577280m^2+445824000m+668860416)/696729600$ \\ \hline
$\pm15$ & $(m^8+120m^7+6020m^6+163800m^5+2616558m^4+24693480m^3+130802580m^2+345011400m+376133625)/696729600$ \\ \hline
$\pm20$ & $(m^8+120m^7+6020m^6+163800m^5+2626008m^4+25260480m^3+141860480m^2+424320000m+544768000)/696729600$ \\ \hline
$30$ & $(m^8+120m^7+6020m^6+163800m^5+2626008m^4+25260480m^3+142577280m^2+440380800m+615081600)/696729600$ \\ \hline
\end{tabularx}
\caption{Results for $N(E_8,m)$}
\label{ne8m}
\end{table}
\end{theorem}

\begin{remark}\label{rem:misprint}
For $G=E_6$ and $m\equiv3\bmod 6$, \cite[Table 4]{Djo85} provided the formula 
\begin{equation}\label{eq:djo}
N(E_6,6k+3)=(k+1)(k+2)(108k^4+688k^3+1395k^2+1269k+480)/120.
\end{equation}
After comparing Equation \eqref{eq:djo} with the relevant formula in Table \ref{ne6m}, we note that the correct coefficient of $k^3$ in the third term is $648$, rather than $688$.
\end{remark}

\FloatBarrier
\section{$N(G,m,s)$}\label{ngms}
As we did with $N(G,m)$, we begin by presenting our methods for $G=G_2$ (Section \ref{subsec:ng2ms}) and arriving at a complete computation for that case. We  then generalize to any exceptional Lie group (Section \ref{subsec:ngms}).

\subsection{Computation of $N(G_2,m,s)$}\label{subsec:ng2ms} We begin (Subsection \ref{g2distincteval}) by answering the following question: when does a typical torus element $t(k_{\delta_1},k_{\delta_2})$ have $s$ distinct eigenvalues? Once this question is answered, we determine (Subsection \ref{g2fixeds}) how many such elements are fixed under the action of an element of $W$. Then, we determine the final result (Subsection \ref{subsubsec:ng2msfinal}). 

\subsubsection{Number of distinct eigenvalues of $t(k_{\delta_1},k_{\delta_2})$}\label{g2distincteval} 

The multiset of eigenvalues of $t(k_{\delta_1},k_{\delta_2})$ is  
\begin{equation} \label{eqn:g2eigenvalues}
\left \{ 1, e^{\pm 2\pi iP_1}, e^{\pm 2\pi iP_2}, e^{\pm 2\pi iP_3}\right \} = \left\{1,e^{\pm 2\pi i k_{\delta_1}},e^{\pm 2\pi i(k_{\delta_1}-k_{\delta_2})},e^{\pm 2\pi i(2k_{\delta_1}-k_{\delta_2})}\right\}.
\end{equation}
We start by making the following observation.

\begin{proposition}\label{prop:g2repeat}
One or more repeats occur in \eqref{eqn:g2eigenvalues} iff
\begin{enumerate}
\item $P_i\pm P_j=0$ for some $i\ne j$ (i.e., $e^{\pm 2\pi iP_i}$ coincide with $e^{\mp 2\pi iP_j}$),
\item $P_i=0$ for some $i$ (i.e., $e^{\pm 2\pi iP_i}=1$), or
\item $2P_i=0$ but $P_i\ne 0$ for some $i$ (i.e., $P_i=\frac{1}{2}$ and $e^{\pm 2\pi iP_i}=-1$).
\end{enumerate}
\end{proposition}

We translate this into the language of matrices and kernels as follows. Corresponding to $P_1$, $P_2$, and $P_3$ (see \eqref{p1p2p3}), we define 
\[\vec{v}_1=\begin{bmatrix} 1 & 0 \end{bmatrix},\ \vec{v}_2=\begin{bmatrix} -1 & 1 \end{bmatrix},\text{ and }\vec{v}_3=\begin{bmatrix} 2 & -1 \end{bmatrix},\]
i.e., $\vec{v_i}$ is the row vector of $P_i$ with respect to the basis $\{ k_{\delta_1},k_{\delta_2} \}$. By stacking together the distinct vectors (up to a sign) in the list \[\left\{\vec{v}_1\pm\vec{v}_2,\vec{v}_1\pm\vec{v}_3,\vec{v}_2\pm\vec{v}_3,\vec{v}_1,\vec{v}_2,\vec{v}_3,2\vec{v}_1,2\vec{v}_2,2\vec{v}_3\right\},\] we form the matrix\footnote{Since we are interested only in the kernels of submatrices of $P$ over $\frac{1}{m}\Z/\Z$, we are free to change the signs and ordering of the rows of $P$ as we see fit. Here, for aesthetic reasons, we require that the first nonzero element in each row of $P$ be positive.} 
\[P=\begin{bmatrix}
1 & 1 & 2 & 2 & 0 & 3 & 2 & 3 & 4 \\
0 & -1 & -1 & 0 & 1 & -1 & -2 & -2 & -2
\end{bmatrix}^T.\] Proposition \ref{prop:g2repeat} is thus recast into the following.

\begin{proposition}\label{prop:g2repeatkernel}
The pair $(k_{\delta_1},k_{\delta_2})^T\in\left(\frac{1}{m}\Z/\Z\right)^2$ leads to one or more repeats in \eqref{eqn:g2eigenvalues} iff $(k_{\delta_1},k_{\delta_2})^T$ is in the kernel of some $k\times 2$ submatrix of $P$, where $k\ge 1$.
\end{proposition}

\begin{example}\label{ex:g2repeat}
Consider $m=3$ and the submatrix \[Q=\begin{bmatrix} 0 & 1 \\ 3 & -1 \\ 3 & -2 \end{bmatrix}\] of $P$. It is easily checked that $\left(\frac{1}{3},0\right)^T\in\left(\frac{1}{3}\Z/\Z\right)^2$ is in the kernel of $Q$. Observe that \[\begin{bmatrix} 0 & 1 \end{bmatrix}=\vec{v}_1+\vec{v}_2,\ \begin{bmatrix} 3 & -1 \end{bmatrix}=\vec{v}_1+\vec{v}_3,\text{ and }\begin{bmatrix} 3 & -2 \end{bmatrix}=-\vec{v}_2+\vec{v}_3.\] By the correspondence between $\vec{v}_i$ and $P_i$, we have that \[P_1\left(\frac{1}{3},0\right)=-P_2\left(\frac{1}{3},0\right),\ P_1\left(\frac{1}{3},0\right)=-P_3\left(\frac{1}{3},0\right),\text{ and }P_2\left(\frac{1}{3},0\right)=P_3\left(\frac{1}{3},0\right).\] 
Thus, in \eqref{eqn:g2eigenvalues} with $(k_{\delta_1},k_{\delta_2})=\left(\frac{1}{3},0\right)$, 
\[e^{\pm 2\pi ik_{\delta_1}},\ e^{\pm 2\pi i(k_{\delta_1}-k_{\delta_2})},\text{ and }e^{\pm 2\pi i(2k_{\delta_1}-k_{\delta_2})}\] all coincide. It can also be verified that $Q$ is the largest submatrix of $P$ having $\left(\frac{1}{3},0\right)^T$ in its kernel. Hence, we have fully captured the manner (in the sense of Proposition \ref{prop:g2repeatkernel}) by which the eigenvalues of $t\left(\frac{1}{3},0\right)$ repeat. We conclude that 
\[\left\{1,e^{\pm 2\pi ik_{\delta_1}}\right\}\] 
is the irredundant list of eigenvalues of $t\left(\frac{1}{3},0\right)$, i.e., $t\left(\frac{1}{3},0\right)$ has three distinct eigenvalues.
\end{example}

Our next step is to generalize the idea used in Example \ref{ex:g2repeat} to a method for determining the number of distinct eigenvalues of $t(k_{\delta_1},k_{\delta_2})$ based on the interaction between $(k_{\delta_1},k_{\delta_2})^T$ and the matrix $P$. For convenience, we will treat $\left(\frac{1}{m}\Z/\Z\right)^2$ as the kernel of the unique $0\times 2$ empty submatrix of $P$.

Let $S$ be a $k\times 2$ submatrix of $P$, where $0\le k\le 9$, and suppose that $(k_{\delta_1},k_{\delta_2})^T\in\left(\frac{1}{m}\Z/\Z\right)^2$ is in the kernel of $S$. If $(k_{\delta_1},k_{\delta_2})^T$ is not in the kernel of any other submatrix of $P$ properly containing $S$, then $S$ completely characterizes the manner (in the sense of Proposition \ref{prop:g2repeat}) by which the eigenvalues of $t(k_{\delta_1},k_{\delta_2})$ repeat. This motivates the imposition of the following partial order on the set of $k\times 2$ submatrices of $P$.
\begin{definition}
Let $\Scal$ be the set of $k\times 2$ submatrices of $P$, where $0\le k\le 9$. We define a partial order $\preceq$ on $\Scal$ by $S_1\preceq S_2$ iff $S_2$ is a submatrix of $S_1$.
\end{definition}

\begin{remark}
The partial order $\preceq$ can be interpreted as ``more restrictive than'' by viewing each $S\in\Scal$ as a set of conditions imposed on the pairs in $\left(\frac{1}{m}\Z/\Z\right)^2$, by which $(k_{\delta_1},k_{\delta_2})^T$ meets the conditions iff $(k_{\delta_1},k_{\delta_2})^T$ is in the kernel of $S$.
\end{remark}

\begin{proposition} \label{prop:g2eigencharacterization}
For each $(k_{\delta_1},k_{\delta_2})^T\in\left(\frac{1}{m}\Z/\Z\right)^2$, there exists a unique minimal element $S\in\Scal$ having $(k_{\delta_1},k_{\delta_2})^T$ in its kernel. 
\end{proposition}
\begin{proof} The minimal element $S$ is the submatrix of $P$ containing all rows annihilating $(k_{\delta_1},k_{\delta_2})^T$. 
\end{proof}

Given $S\in\Scal$ and $i,j\in \{1,2,3\}$, define an equivalence relation $\equiv_S$ on $H=\left\{0,\frac{1}{2},1,2,3\right\}$ by the rules:
 
\begin{enumerate}
\item $i\equiv_S j$ if $\vec{v}_i\pm\vec{v}_j$ is (up to a sign) a row vector of $S$;
\item $i\equiv_S 0$ if $\vec{v}_i$ is (up to a sign) a row vector of $S$; and
\item $i\equiv_S \frac{1}{2}$ if $2\vec{v}_i$ is (up to a sign) a row vector of $S$ but $\vec{v}_i$ is not.
\end{enumerate}
The number of distinct eigenvalues of $t(k_{\delta_1},k_{\delta_2})$ can then be expressed in terms of the number of equivalence classes of the set $H$ under $\equiv_S$.

\begin{proposition} \label{prop:g2sfunction}
Let $(k_{\delta_1},k_{\delta_2})^T\in\left(\frac{1}{m}\Z/\Z\right)^2$ and let $S\in\Scal$ be the minimal element having $(k_{\delta_1},k_{\delta_2})^T$ in its kernel. Then, the number of distinct eigenvalues of $t(k_{\delta_1},k_{\delta_2})$ is given by 
\begin{equation} \label{eqn:g2sfunction}
s(S)=
\begin{cases}
2\left\lvert H/\equiv_S\right\rvert-3 & \mbox{if }\left\lvert\left[\frac{1}{2}\right]\right\rvert=1 \\
2\left\lvert H/\equiv_S\right\rvert-2 & \mbox{if } \left\lvert\left[\frac{1}{2}\right]\right\rvert>1\\
\end{cases}.
\end{equation}
\end{proposition}
\begin{proof} Each of the equivalence classes $[1]$, $[2]$, and $[3]$, when they do not coincide with $[0]$ or $\left[\frac{1}{2}\right]$, contributes two distinct eigenvalues, giving $2\left (\left\lvert H/\equiv_S\right\rvert-2\right )$. The equivalence class $[0]$ contributes the eigenvalue 1 which is present for all $(k_{\delta_1},k_{\delta_2})^T$. The equivalence class $\left[\frac{1}{2}\right]$ contributes the eigenvalue $-1$ exactly when $\left \lvert \left[\frac{1}{2}\right]\right\rvert>1$. The result follows. 
\end{proof}

\begin{example}
Referring back to Example \ref{ex:g2repeat}, we see that \[\left\{0\right\},\left\{\frac{1}{2}\right\},\left\{1,2,3\right\}\] are the equivalence classes in $H$ under $\equiv_Q$, so $s(Q)=3$. On the other hand, when $m=3$, $\left(\frac{1}{3},0\right)^T\in\left(\frac{1}{m}\Z/\Z\right)^2$ is in the kernel of $Q$ but not that of any other submatrix of $P$ properly containing $Q$, and $t\left(\frac{1}{3},0\right)$ has three distinct eigenvalues.
\end{example}

\subsubsection{Number of torus elements with $s$ eigenvalues fixed by an element of $W$}\label{g2fixeds}

We begin by establishing a correspondence between $W$ and certain submatrices of $P$, whereby $t(k_{\delta_1},k_{\delta_2})$ is fixed by $w\in W$ iff $(k_{\delta_1},k_{\delta_2})^T$ falls in the kernel of the corresponding submatrix $S_w$. Recall that, ignoring the common factor of $2\pi i$, the exponents of the first three diagonal entries of $t(k_{\delta_1},k_{\delta_2})$ are 
\[(P_1,P_2,P_3).\] 
as given in Equation \eqref{p1p2p3}.
Let $\sigma_w$ be the signed permutation on $\{1,2,3\}$ that is associated to $w$. For $i=1,2,3$, we define
\[\sigma_w(P_i)=\begin{cases} P_{\sigma_w(i)} & \sigma_w(i)>0 \\
-P_{-\sigma_w(i)} & \sigma_w(i)<0 \end{cases}.\]
Thus, the exponents of the first three diagonal entries of $w\cdot t(k_{\delta_1},k_{\delta_2})$ are, once again ignoring the common factor of $2\pi i$,
\[(\sigma_w(P_1),\sigma_w(P_2),\sigma_w(P_3)).\]
Now, $t(k_{\delta_1},k_{\delta_2})$ and $w\cdot t(k_{\delta_1},k_{\delta_2})$ agree in their $i$th diagonal entries, for $i=1,2$, iff one of the following holds:
\begin{enumerate}
\item $\sigma_w(i)=i$;
\item $\sigma_w(i)=-i$ and $2P_i=0$, i.e., $(k_{\delta_1},k_{\delta_2})^T$ is in the kernel of $2\vec{v}_i$;
\item $\sigma_w(i)=\pm j$, where $j\in\{1,2,3\}$ and $j\ne i$, and $P_i\mp P_j=0$, i.e., $(k_{\delta_1},k_{\delta_2})^T$ is in the kernel of $\vec{v}_i\mp\vec{v}_j$. 
\end{enumerate}
These conditions lead to the following definition of the matrix $S_w$ associated with $w\in W$.
\begin{definition}  Let $w\in W$ and let $\sigma _w$ be the signed permutation associated to $w$. Then, $S_w\in \Scal$ is the (possibly empty) matrix obtained by stacking together the (distinct up to a sign) rows of $P$ in the following:
\begin{enumerate}
\item $2\vec{v}_i$, if $\sigma_w(i)=-i$, for $i=1,2$;
\item $\vec{v}_i\mp\vec{v}_j$, if $\sigma_w(i)=\pm j$, for $i=1,2$, $j\in\{1,2,3\}$, and $j\ne i$.
\end{enumerate}
\end{definition}

\begin{proposition} \label{prop:g2fixedpointcondition}
The torus element $t(k_{\delta_1},k_{\delta_2})$ is fixed by $w$ iff $(k_{\delta_1},k_{\delta_2})^T$ is in the kernel of $S_w$.
\end{proposition}

\begin{remark}
The matrices $S_w$, where $w\in W$ is a conjugacy class representative of our choice, have almost been computed in Table \ref{tab:fixedpointg2} and used in Example \ref{ex:g2fixedpoint}. The only modification we have made here is that we are excluding the rows of zeros and repetitions that may arise from comparing $t(k_{\delta_1},k_{\delta_2})$ and $w\cdot t(k_{\delta_1},k_{\delta_2})$ in their first two diagonal entries. See Table \ref{tab:g2sw}.
\end{remark}

\begin{table}[h]
\centering
\begin{tabular}{|c|c|c|}
\hline
$\lvert c\rvert$ & $w_c\in c$ & $S_{w_c}$ \\ \hline
$1$ & $I$ & the empty submatrix \\ \hline
$3$ & $n_{\delta_1}(1)$ & $\begin{bmatrix}2 & -1\end{bmatrix}$ \\ \hline
$3$ & $n_{\delta_2}(1)$ & $\begin{bmatrix}3 & -2 \end{bmatrix}$ \\ \hline
$2$ & $n_{\delta_2}(1)n_{\delta_1}(1)$ & $\begin{bmatrix}2 & -1 \\ 1 & 0\end{bmatrix}$ \\ \hline
$2$ & $\left(n_{\delta_2}(1)n_{\delta_1}(1)\right)^2$ & $\begin{bmatrix}3 & -1 \\ 0 & 1\end{bmatrix}$ \\ \hline
$1$ & $\left(n_{\delta_2}(1)n_{\delta_1}(1)\right)^3$ & $\begin{bmatrix}2 & 0 \\ 2 & -2\end{bmatrix}$ \\ \hline
\end{tabular}
\caption{The matrices $S_{w_c}$, where $w_c$ is a conjugacy class representative of our choice.}
\label{tab:g2sw}
\end{table}

Equipped with the means to detect when $t(k_{\delta_1},k_{\delta_2})$ has $s$ distinct eigenvalues or is fixed by some $w\in W$, we are ready to study the number of fixed points in $E(T,m,s)$ under the action of each $w\in W$. The approach that we take involves repeated exploitations of the partial order on $\Scal$.

For each $S\in\Scal$, let \begin{align*}
G_m(S)&=\left\{(k_{\delta_1},k_{\delta_2})^T\in\left(\frac{1}{m}\Z/\Z\right)^2:\ S(k_{\delta_1},k_{\delta_2})^T=\vec{0}\right\},\text{ and} \\
F_m(S)&=\left\{(k_{\delta_1},k_{\delta_2})^T\in\left(\frac{1}{m}\Z/\Z\right)^2:\ S(k_{\delta_1},k_{\delta_2})^T=\vec{0},\ T(k_{\delta_1},k_{\delta_2})^T\ne\vec{0}\text{ for all }T\prec S\right\},
\end{align*}
and define, accordingly, 
\begin{equation}\label{eqn:gm-fm-def}
g_m(S)=\left\lvert G_m(S)\right\rvert,\text{ and }f_m(S)=\left\lvert F_m(S)\right\rvert.
\end{equation}

It follows from Proposition \ref{prop:g2eigencharacterization} that we have the partition 
\begin{equation} \label{eqn:g2partition}
E(T,m,s)=\bigcup_{\substack{S\in \Scal \\ s(S)=s}}\left\{t(k_{\delta_1},k_{\delta_2}):\ (k_{\delta_1},k_{\delta_2})^T\in F_m(S)\right\}.
\end{equation}

\begin{theorem} \label{prop:g2allornothing}
If $F_m(S)\ne\emptyset$, then each $w\in W$ fixes either all or none of the $t(k_{\delta_1},k_{\delta_2})$ for $(k_{\delta_1},k_{\delta_2})^T\in F_m(S)$. In particular, $w$ fixes all such elements iff $S\preceq S_w$.
\end{theorem}

\begin{proof}
Suppose that $F_m(S)\ne\emptyset$, and fix $w\in W$. Suppose that for some $(k_{\delta_1},k_{\delta_2})^T\in F_m(S)$, $t(k_{\delta_1},k_{\delta_2})$ is fixed by $w$. Then, by Proposition \ref{prop:g2fixedpointcondition}, $(k_{\delta_1},k_{\delta_2})^T$ is in the kernel of $S_w\in\Scal$. However, $S$ is the unique minimal element of $\Scal$ having $(k_{\delta_1},k_{\delta_2})^T$ in its kernel, so $S\preceq S_w$. Hence, the kernel of $S$ is contained inside that of $S_w$. By Proposition \ref{prop:g2fixedpointcondition}, this implies that $w$ fixes all of the $t(k_{\delta_1},k_{\delta_2})$, where $(k_{\delta_1},k_{\delta_2})^T\in F_m(S)$.
\end{proof}

It follows that to determine the number of fixed points in $E(T,m,s)$ under the action of each $w\in W$, it makes practical sense for us to compute $f_m(S)$ in the event that $F_m(S)\ne\emptyset$. We introduce a necessary condition for $S\in\Scal$ to satisfy $F_m(S)\ne\emptyset$.

\begin{proposition} \label{prop:g2optimization}
Let $S\in\Scal$. If there exists a row vector $\vec{v}$ of $P$ that (a) is an integer combination of the row vectors of $S$ but (b) is not a row vector of $S$, then $f_m(S)=0$.
\end{proposition}

\begin{proof}
Suppose that such a $\vec{v}$ exists. Let $S'$ be the smallest submatrix of $P$ containing both $S$ and $\vec{v}$. Since $\vec{v}$ is an integer combination of the row vectors of $S$, $\ker S$ and $\ker S'$ over $\frac{1}{m}\Z/\Z$ coincide. It follows from the definition that $F_m(S)=\emptyset$, and so $f_m(S)=0$.
\end{proof}

Let $\Mcal$ denote the subset of $\Scal$ that consists of the matrices in $\Scal$ satisfying the following property: every row vector of $P$ that is an integer combination of the row vectors of $S$ is a row vector of $S$. Then, by Proposition \ref{prop:g2optimization}, $S\in \Mcal$ whenever $F_m(S)\ne\emptyset$. Note that $\Mcal$ inherits the partial order from $\Scal$.

Observe that, for all $S\in\Scal$, 
\begin{equation}\label{gmsum}
g_m(S)=\sum_{T\preceq S}f_m(T)=\sum_{\substack{T\preceq S \\ T\in\Mcal}}f_m(T).
\end{equation}
 Then, it follows from the M\"obius  inversion formula for posets that, for all $S\in\Mcal$,
\begin{equation} \label{eqn:g2fgrelation}
f_m(S)=\sum_{\substack{T\preceq S\\ T\in\Mcal}}\mu (T,S)g_m(T),
\end{equation}
where $\mu$ is the M\"obius function of $\Mcal$ defined as
\begin{equation} \label{eqn:mobiusfunction}
\mu(T,S)=\begin{cases} 1, & \mbox{if }T=S \\ -\sum_{T\preceq U\prec S}\mu(T,U), & \mbox{if }T\prec S \\ 0, & \mbox{else} \end{cases}.
\end{equation}

With respect to computing $f_m(S)$, it remains to determine $g_m(T)$ and $\mu (T,S)$ for each $T\preceq S$ with $T\in\Mcal$. However, by definition, $g_m(T)$ is the size of the kernel of $T$ over $\frac{1}{m}\Z/\Z$, and so it can be directly computed using Proposition \ref{prop:countkernel}. Thus, in conjunction with Proposition \ref{prop:countkernel}, Equation \eqref{eqn:g2fgrelation} provides us with a practical method for computing $f_m(S)$.

In consideration of the partition in \eqref{eqn:g2partition}, the number of fixed points of $E(T,m,s)$ under the action of some $w\in W$ is the sum of the number of elements in each subset on the right-hand side that are fixed by $w$. In view of Theorem \ref{prop:g2allornothing}, we have 
\begin{equation} \label{eqn:g2msfixedpoint}
\left\lvert\Fix_s(w)\right\rvert=\sum_{\substack{S\in\Mcal \\ s(S)=s \\ S\preceq S_{w}}}f_m(S).
\end{equation}
Equation \eqref{eqn:g2msfixedpoint} combines with Theorem \ref{thm:burnsidereformulated} to give the following formula. 
\begin{corollary} \label{cor:ng2ms}We have 
\begin{equation} 
N(G_2,m,s)=\frac{1}{\lvert W\rvert}\sum_{c\in\Cl(W)}\left\lvert c\right\rvert\sum_{\substack{S\in\Mcal \\ s(S)=s \\ S\preceq S_{w_c}}}f_m(S),
\end{equation}
where $f_m(S)$ can be computed using Equation \eqref{eqn:g2fgrelation}.
\end{corollary}

\subsubsection{Completing the computation of $N(G_2,m,s)$}\label{subsubsec:ng2msfinal}

We introduce a simple optimization for implementing Corollary \ref{cor:ng2ms}.

\begin{definition}[Row lattice generated by a matrix]
Let $A\in\M_{m\times n}(\Z)$. The \emph{row lattice} generated by $A$ is the set \[\Lambda_R(A):=\left\{\vec{v}^TA:\vec{v}\in\Z^m\right\}.\]
\end{definition}

\begin{fact} \label{fact:g2fact}
For all $S\in\Scal$, there exists a $k\times 2$ submatrix $S'$ of $P$, where $0\le k\le 2$, such that $\Lambda_R(S)=\Lambda_R(S')$. See Algorithm \ref{alg:findp} in the appendix for a direct way to check this.
\end{fact}

Let $\row(P)$ denote the set of row vectors of $P$. Given any $S\in\Mcal$, the set of row vectors of $S$ coincides with $\Lambda_R(S)\cap\row(P)$ by construction. Conversely, given any $S'\in\Scal$, the submatrix $S$ of $P$ with the set of row vectors $\Lambda_R(S')\cap\row(P)$ is necessarily in $\Mcal$. Hence, $\Mcal$ consists of the matrices in $\Scal$ having the set of row vectors $\Lambda_R(S)\cap\row(P)$ for some $S\in\Scal$. With Fact \ref{fact:g2fact}, we deduce the following.

\begin{proposition} \label{prop:g2mmatrices}
The set $\Mcal$ consists of the matrices in $\Scal$ having the set of row vectors $\Lambda_R(S)\cap\row(P)$, where $S\in\Scal$ has no more than two rows.
\end{proposition}

It turns out that there are $19$ matrices in $\Mcal$, which we list in Table \ref{tab:g2mcal} along with their $s$- and $g_m$-values. With this information, we are completely prepared to compute $N(G_2,m,s)$.

\begin{table}[h]
\centering
\begin{tabular}{|c|c|c|}
\hline
$S\in\Mcal$ & $s(S)$ & $g_m(S)$ \\ \hline

the empty submatrix & $7$ & $m^2$ \\ \hline

$\begin{bmatrix}2 & 0\end{bmatrix}$, $\begin{bmatrix}2 & -2\end{bmatrix}$, $\begin{bmatrix}4 & -2\end{bmatrix}$ & $6$ & $\begin{cases} 2m & m\equiv 0\bmod 2 \\ m & \mbox{else} \end{cases}$ \\ \hline

$\begin{bmatrix}3 & -2\end{bmatrix}$, $\begin{bmatrix}3 & -1\end{bmatrix}$, $\begin{bmatrix}0 & 1\end{bmatrix}$ & $5$ & $m$ \\ \hline

$\begin{bmatrix}0 & 1 \\ 4 & -2\end{bmatrix}$, $\begin{bmatrix}3 & -2 \\ 2 & 0\end{bmatrix}$, $\begin{bmatrix}3 & -1 \\ 2 & -2\end{bmatrix}$ & $4$ & $\begin{cases} 4 & m\equiv 0\bmod 4 \\ 1 & m\equiv 1,3\bmod 4 \\ 2 & \mbox{else} \end{cases}$ \\ \hline

$\begin{bmatrix}1 & 0 \\ 2 & 0\end{bmatrix}$, $\begin{bmatrix}1 & -1 \\ 2 & -2\end{bmatrix}$, $\begin{bmatrix}2 & -1 \\ 4 & -2\end{bmatrix}$ & $3$ & $m$ \\ \hline

$\begin{bmatrix}3 & 3 & 0 \\ -2 & -1 & 1\end{bmatrix}^T$ & $3$ & $\begin{cases} 3 & m\equiv 0\bmod 3 \\ 1 & \mbox{else} \end{cases}$ \\ \hline

$\begin{bmatrix}2 & 2 & 4 \\ 0 & -2 & -2\end{bmatrix}^T$ & $2$ & $\begin{cases} 4 & m\equiv 0\bmod 2 \\ 1 & \mbox{else} \end{cases}$ \\ \hline

$\begin{matrix}
\begin{bmatrix}2 & 0 & 2 & 2 & 4 \\ -1 & 1 & 0 & -2 & -2\end{bmatrix}^T, \\ \begin{bmatrix}1 & 3 & 2 & 2 & 4 \\ 0 & -2 & 0 & -2 & -2\end{bmatrix}^T, \\ \begin{bmatrix}1 & 3 & 2 & 2 & 4 \\ -1 & -1 & 0 & -2 & -2\end{bmatrix}^T
\end{matrix}$ & $2$ & $\begin{cases} 2 & m\equiv 0\bmod 2 \\ 1 & \mbox{else} \end{cases}$ \\ \hline

$\begin{bmatrix} 1 & 1 & 2 & 2 & 0 & 3 & 2 & 3 & 4 \\ 0 & -1 & -1 & 0 & 1 & -1 & -2 & -2 & -2 \end{bmatrix}^T$ & $1$ & $1$ \\ \hline
\end{tabular}
\caption{The $19$ matrices in $\Mcal$.}
\label{tab:g2mcal}
\end{table}

\begin{example}[Computing $N(G_2,m,2)$]
For convenience, let 
\begin{align*}
S_1&=\begin{bmatrix}2 & 2 & 4 \\ 0 & -2 & -2\end{bmatrix}^T, \\
S_2&=\begin{bmatrix}2 & 0 & 2 & 2 & 4 \\ -1 & 1 & 0 & -2 & -2\end{bmatrix}^T, \\
S_3&=\begin{bmatrix}1 & 3 & 2 & 2 & 4 \\ 0 & -2 & 0 & -2 & -2\end{bmatrix}^T, \\
S_4&=\begin{bmatrix}1 & 3 & 2 & 2 & 4 \\ -1 & -1 & 0 & -2 & -2\end{bmatrix}^T,\text{ and }\\
S_5&=P=\begin{bmatrix} 1 & 1 & 2 & 2 & 0 & 3 & 2 & 3 & 4 \\ 0 & -1 & -1 & 0 & 1 & -1 & -2 & -2 & -2 \end{bmatrix}^T.
\end{align*}
We begin by computing the M\"obius coefficients using Equation \eqref{eqn:mobiusfunction}:
\begin{align*}
\mu(S_5,S_4)&=\mu(S_5,S_3)=\mu(S_5,S_2)=\mu(S_4,S_1)=\mu(S_3,S_1)=\mu(S_2,S_1)=-1,\text{ and } \\
\mu(S_5,S_1)&=-\left[\mu(S_5,S_5)+\mu(S_5,S_4)+\mu(S_5,S_3)+\mu(S_5,S_2)\right]=2.
\end{align*}
Then, we can compute $f_m(S_i)$ for $i=1,\dots,4$ using Equation \eqref{eqn:g2fgrelation}:
\begin{align*}
f_m(S_1)&=\sum_{i=1}^5\mu(S_i,S_1)g_m(S_i)=\begin{cases} 4 & m\equiv 0\bmod 2 \\ 1 & \mbox{else} \end{cases}-3\cdot\begin{cases} 2 & m\equiv 0\bmod 2 \\ 1 & \mbox{else} \end{cases}+2=0,\text{ and } \\
f_m(S_2)&=\mu(S_5,S_2)g_m(S_5)+\mu(S_2,S_2)g_m(S_2)=\begin{cases} 1 & m\equiv 0\bmod 2 \\ 0 & \mbox{else} \end{cases}=f_m(S_3)=f_m(S_4).
\end{align*}
Since $f_m(S_1)=0$, it can be neglected for the purpose of applying Corollary \ref{cor:ng2ms}. To check for fixed points, we observe that, by referring to Table \ref{tab:g2sw},
\begin{enumerate}
\item $S_2,S_3,S_4\preceq S_{I},S_{(n_{\delta_2}(1)n_{\delta_1}(1))^3}$,
\item $S_2\preceq S_{n_{\delta_1}(1)}$,\text{ and }
\item $S_3\preceq S_{n_{\delta_2}(1)}$.
\end{enumerate}
Hence, \begin{align*}
N(G_2,m,2)&=\frac{1}{12}\left[(1+1)\cdot(f_m(S_2)+f_m(S_3)+f_m(S_4))+3\cdot f_m(S_2)+3\cdot f_m(S_3)\right] \\
&=\begin{cases} 1 & m\equiv 0\bmod 2 \\ 0 & \mbox{else} \end{cases}.
\end{align*}
\end{example}

Repeating this for the other possible values of $s$, we obtain the following theorem.

\begin{theorem}
For any positive integer $m$ and $1\le s\le 7$, the number $N(G_2,m,s)$ is given by Table \ref{tab:Ng2ms}.
\begin{table}
\centering
\def\arraystretch{1.6}
\begin{tabular}{|c|c|c|c|c|c|c|c|}
\hline
$N(G_2,m,s)$ & \multicolumn{7}{c|}{$s$} \\ \hline
$m\bmod 12$ & $1$ & $2$ & $3$ & $4$ & $5$ & $6$ & $7$ \\ \hline
$0$ & $1$ & $1$ & $m/2$ & $1$ & $(m-6)/2$ & $(m-4)/4$ & $m(m-9)/12+2$ \\ \hline
$1,5,7,11$ & $1$ & $0$ & $(m-1)/2$ & $0$ & $(m-1)/2$ & $0$ & $(m-1)(m-5)/12$ \\ \hline
$2,10$ & $1$ & $1$ & $(m-2)/2$ & $0$ & $(m-2)/2$ & $(m-2)/4$ & $(m-2)(m-7)/12$ \\ \hline
$3,9$ & $1$ & $0$ & $(m+1)/2$ & $0$ & $(m-3)/2$ & $0$ & $(m-3)^2/12$ \\ \hline
$4,8$ & $1$ & $1$ & $(m-2)/2$ & $1$ & $(m-4)/2$ & $(m-4)/4$ & $(m-4)(m-5)/12$ \\ \hline
$6$ & $1$ & $1$ & $m/2$ & $0$ & $(m-4)/2$ & $(m-2)/4$ & $(m-3)(m-6)/12$ \\ \hline
\end{tabular}
\caption{Results for $N(G_2,m,s)$.}
\label{tab:Ng2ms}
\end{table}
\end{theorem}

\subsection{Computation of $N(G,m,s)$ for all exceptional Lie groups.}\label{subsec:ngms}

The approach we take here directly extends the one in Section \ref{subsec:ng2ms}. We invite the reader to refer back to Section \ref{subsec:ng2ms} for motivation and proofs.

\subsubsection{Number of distinct eigenvalues of $t(\vec{k})$}
First, we develop a systematic method to determine the number of distinct eigenvalues of $t(\vec{k})$ for all $\vec{k}\in\left(\frac{1}{m}\Z/\Z\right)^\ell$. Using the typical element of the torus given near the beginning of Section \ref{subsec:ngm}, we have that the multiset of eigenvalues of $t(\vec{k})$ is 
\begin{equation} \label{eqn:eigenvalues}
\begin{aligned}
\left\{1,\dots,1,e^{\pm 2\pi iP_1},\dots,e^{\pm 2\pi iP_u}\right\}, & \mbox{ if }G=G_2,F_4,E_8; \\
\left\{e^{2\pi iP_1},\dots,e^{2\pi iP_u}\right\}, & \mbox{ if }G=E_6;\text{ or} \\
\left\{e^{\pm 2\pi iP_1},\dots,e^{\pm 2\pi iP_u}\right\}, & \mbox{ if }G=E_7.
\end{aligned}
\end{equation}

We begin by noting when repeats (besides those among $1,\dots,1$ in the first case) occur in \eqref{eqn:eigenvalues}.

\begin{proposition}\label{prop:repeat}
The $\ell$-tuple  $\vec{k}\in\left(\frac{1}{m}\Z/\Z\right)^\ell$ introduces (additional) repeats to \eqref{eqn:eigenvalues} iff
\begin{enumerate}
\item 
\begin{enumerate}
\item $P_i\pm P_j=0$ for some $i\ne j$, or
\item $P_i=0$ for some $i$, or
\item $2P_i=0$ but $P_i\ne 0$ for some $i$,
\end{enumerate}
for $G=G_2,F_4,E_7,E_8$.
\item  $P_i=P_j$ for some $i\ne j$,

for $G=E_6$.
\end{enumerate}
\end{proposition}

To translate this into the language of matrices and kernels, we form a matrix $P$ by stacking together the distinct vectors (up to a sign) in the list 
\[\left\{\vec{v}_1\pm\vec{v}_2,\vec{v}_1\pm\vec{v}_3,\dots,\vec{v}_{u-1}\pm\vec{v}_u,\vec{v}_1,\dots,\vec{v}_u,2\vec{v}_1,\dots,2\vec{v}_u\right\},\]
where as before $\vec{v}_i$ is the row vector of $P_i$ with respect to the basis $\{k_{\delta_1},\ldots,k_{\delta_\ell} \}$.\footnote{While the repeats of eigenvalues for $t(\vec{k})\in E_6$ can be accounted for with only the vectors in the set $\left\{\vec{v}_1-\vec{v}_2,\vec{v}_1-\vec{v}_3,\dots,\vec{v}_{u-1}-\vec{v}_u\right\}$, doing so will not provide sufficient information for determining whether $t(\vec{k})$ is fixed by some $w\in W$.}

\begin{proposition}
The $\ell$-tuple $\vec{k}\in\left(\frac{1}{m}\Z/\Z\right)^\ell$ introduces (additional) repeats in \eqref{eqn:eigenvalues} iff, 
\begin{enumerate}
\item for $G=G_2,F_4,E_7,E_8$, $\vec{k}$ is in the kernel of some $k\times\ell$ submatrix of $P$, where $k\ge 1$, or, 
\item for $G=E_6$, $\vec{k}$ is in the kernel of some $k\times\ell$ submatrix of $P$, where $k\ge 1$, that contains a row vector of the form $\vec{v}_i-\vec{v}_j$, for some $i\ne j$.
\end{enumerate}
\end{proposition}

\noindent Once again, for convenience, we treat $\left(\frac{1}{m}\Z/\Z\right)^\ell$ as the kernel of the unique $0\times\ell$ empty submatrix of $P$.

\begin{definition}\label{def:sfunction}
Let $\Scal$ be the set of all (possibly empty) submatrices of $P$ with $\ell$ columns. Define on $\Scal$ a partial order $\preceq$ by $S_1\preceq S_2$ iff $S_2$ is a submatrix of $S_1$. 
\end{definition}

\begin{proposition} \label{prop:eigencharacterization}
For all $\vec{k}\in\left(\frac{1}{m}\Z/\Z\right)^\ell$, there exists a unique minimal element $S\in\Scal$ having $\vec{k}$ in its kernel. 
\end{proposition}

For $G=G_2,F_4,E_7,E_8$, 
let $S\in\Scal$, let $i,j\in \{1,2,\ldots , u\}$, let \[H=\left\{0,\frac{1}{2},1,\dots,u\right\},\] 
and define an equivalence relation $\equiv_S$ on $H$ by the rules:
\begin{enumerate}
\item $i\equiv_S j$ if $\vec{v}_i\pm\vec{v}_j$ is (up to a sign) a row vector of $S$,
\item $i\equiv_S 0$ if $\vec{v}_i$ is (up to a sign) a row vector of $S$, and
\item $i\equiv_S \frac{1}{2}$ if $2\vec{v}_i$ is (up to a sign) a row vector of $S$ but $\vec{v}_i$ is not.
\end{enumerate}
For $G=E_6$, let $S\in\Scal$, let $i,j\in \{1,2,\ldots , u\}$, let \[H=\left\{1,\dots,u\right\},\] and define an equivalence relation $\equiv_S$ on $H$ by the rule:
\begin{enumerate}
\item $i\equiv_S j$ if $\vec{v}_i-\vec{v}_j$ is (up to a sign) a row vector of $S$.
\end{enumerate}

\begin{proposition} \label{prop:sfunction}
Let $\vec{k}\in \left(\frac{1}{m}\Z/\Z\right)^\ell$ and let $S\in\Scal$ be the minimal element having $\vec{k}$ in its kernel. Then, the number of distinct eigenvalues of $t(\vec{k})$ is given by $s(S)$ where

\begin{align}
s(S)&=\begin{cases}  \label{eqn:sfunctiong2f4e8}
2\left\lvert H/\equiv_S\right\rvert-3 & \mbox{if }\left\lvert\left[\frac{1}{2}\right]\right\rvert=1  \\
2\left\lvert H/\equiv_S\right\rvert-2 & \mbox{if }\left\lvert\left[\frac{1}{2}\right]\right\rvert>1
\end{cases},
\text{ for } G=G_2,F_4,E_8; \\ &\nonumber \\
 s(S)&=\left\lvert H/\equiv_S\right\rvert, \label{eqn:sfunctione6}
\text{ for } G=E_6; \text{ and} \\ &\nonumber \\
s(S)&=\begin{cases} \label{eqn:sfunctione7}
2\left\lvert H/\equiv_S\right\rvert-4 & \mbox{if }\left\lvert\left[0\right]\right\rvert=1\text{ and }\left\lvert\left[\frac{1}{2}\right]\right\rvert=1 \\
2\left\lvert H/\equiv_S\right\rvert-2 & \mbox{if }\left\lvert\left[0\right]\right\rvert>1\text{ and }\left\lvert\left[\frac{1}{2}\right]\right\rvert>1 \\
2\left\lvert H/\equiv_S\right\rvert-3 & \mbox{if }\left\lvert\left[0\right]\right\rvert=1\text{ and }\left\lvert\left[\frac{1}{2}\right]\right\rvert>1 
\end{cases},
\text{ for } G=E_7.
\end{align}
\end{proposition}

\subsubsection{Number of torus elements with $s$ eigenvalues fixed by an element of $W$}\label{gfixeds}

We now derive a necessary and sufficient condition for any $t(\vec{k})$ to be fixed by some $w\in W$.

\begin{definition}\label{def:sw}
Given $w\in W$, let $\sigma_w$ be the signed permutation on $[u]=\{1,\dots,u\}$ associated to $w$ by Proposition \ref{prop:wscorrespondence}. We associate to $w$ the (possibly empty) matrix $S_w\in\Scal$, obtained by stacking together the distinct vectors (up to a sign) found via the following procedure, for $i=i_1,\dots,i_\ell$:
\begin{enumerate}
\item if $\sigma_w(i)=-i$, stack $2\vec{v}_i$;
\item if $\sigma_w(i)=\pm j$, where $j\in [u]$ and $j\ne i$, stack $\vec{v}_i\mp\vec{v}_j$. 
\end{enumerate}
\end{definition}

\begin{proposition} \label{prop:fixedpointcondition}
Let $w\in W$. Then, $t(\vec{k})$ is fixed by $w$ iff $\vec{k}$ is in the kernel of $S_w$.
\end{proposition}

As before, for each $S\in\Scal$, let
\begin{align*}
G_m(S)&=\left\{\vec{k}\in\left(\frac{1}{m}\Z/\Z\right)^\ell:\ S\vec{k}=\vec{0}\right\},\text{ and} \\
F_m(S)&=\left\{\vec{k}\in\left(\frac{1}{m}\Z/\Z\right)^\ell:\ S\vec{k}=\vec{0},\ T\vec{k}\ne\vec{0}\text{ for all }T\prec S\right\},
\end{align*}
and define, accordingly, \[g_m(S)=\left\lvert G_m(S)\right\rvert,\text{ and }f_m(S)=\left\lvert F_m(S)\right\rvert.\]
It follows from Proposition \ref{prop:eigencharacterization} that we have the partition 
\begin{equation} \label{eqn:partition}
E(T,m,s)=\bigcup_{\substack{S\in \Scal \\ s(S)=s}}\left\{t(\vec{k}):\ \vec{k}\in F_m(S)\right\}.
\end{equation}

\begin{theorem} \label{thm:allornothing}
If $F_m(S)\ne\emptyset$, then each $w\in W$ fixes either all or none of the $t(\vec{k})$ for $\vec{k}\in F_m(S)$. In particular, $w$ fixes all such elements iff $S\preceq S_w$.
\end{theorem}

\begin{proposition} \label{prop:optimization}
Let $S\in\Scal$. If there exists a row vector $\vec{v}$ of $P$ that (a) is in $\Lambda_R(S)$ but (b) is not a row vector of $S$, then $f_m(S)=0$.
\end{proposition}

As before, we let $\Mcal\subseteq\Scal$ be the partially ordered subset of $\Scal$ consisting of the matrices $S$ such that every row vector of $P$ that is in $\Lambda_R(S)$ is a row vector of $S$. It follows from Proposition \ref{prop:optimization} that if $F_m(S)\neq0$ then $S\in \Mcal$.
Furthermore, Equations \eqref{gmsum}, \eqref{eqn:g2fgrelation}, and \eqref{eqn:g2msfixedpoint} hold here as well. The following corollary therefore holds.

\begin{corollary} \label{cor:ngms}We have 
\begin{equation} \label{eqn:ngms}
N(G,m,s)=\frac{1}{\lvert W\rvert}\sum_{c\in\Cl(W)}\left\lvert c\right\rvert\left\lvert\Fix_s(w_c)\right\rvert,
\end{equation}
where 
\begin{equation}
\left\lvert\Fix_s(w_c)\right\rvert=\sum_{\substack{S\in\Mcal \\ s(S)=s \\ S\preceq S_{w_c}}}f_m(S),
\end{equation} and $f_m(S)$ can be computed using Equation \eqref{eqn:g2fgrelation}.
\end{corollary}

A summary of the procedure for computing $\lvert\Fix_s(w_c)\rvert$ explicitly can be found in Algorithm \ref{alg:fixedpoints} in the appendix.

\subsubsection{Completing the computation of $N(F_4,m,s)$}\label{subsubsec:ngmsf4}

In the case of $G=F_4$, the following analogue of Fact \ref{fact:g2fact} applies. 

\begin{fact} \label{fact:f4fact}
If $G=F_4$, then for all $S\in\Scal$, there exists a $k\times 4$ submatrix $S'$, where $0\le k\le 4$, such that $\Lambda_R(S)=\Lambda_R(S')$.
\end{fact}

As a side note, this fact is based on the following observation.

\begin{proposition}
Let $P\in\M_{n\times l}(\Z)$. If every row lattice generated by a submatrix of $P$ with $l$ columns and no more than $r+1$ rows can be generated by a submatrix with $l$ columns and no more than $r$ rows, then every lattice generated by a submatrix of $P$ with $r$ rows can be generated by a submatrix with $l$ columns and no more than $r$ rows.
\end{proposition}

\noindent Recall that the row-style Hermite normal form (RHNF) allows us to check whether two matrices generate the same row lattice. Hence, we can use the brute-force method described in Algorithm \ref{alg:findp} in the appendix to find the minimum $r$ that will enable us to generate all the matrices in $\Mcal$.

\begin{proposition} \label{prop:f4mmatrices} For $G=F_4$, the set
$\Mcal$ consists of the matrices in $\Scal$ having the set of row vectors $\Lambda_R(S)\cap\row(P)$, where $S\in\Scal$ has no more than 4 rows. 
\end{proposition}
Using Proposition \ref{prop:f4mmatrices}, we find that $\lvert\Mcal \rvert=22075$ for $F_4$. The full table for $N(F_4,m,s)$ can then be computed using Equation \eqref{eqn:ngms} with a computer, which turns out to be given in terms of $m\bmod 12252240$.
Here, the period in $m$ has the factorization $2^{4}\cdot 3^{2}\cdot 5\cdot 7\cdot 11\cdot 13\cdot 17$, the size of which reflects the rich variety of the elementary divisors of the matrices in $\Mcal$; cf. Corollary \ref{cor:ngms}, Equation \eqref{eqn:gm-fm-def}, and Proposition \ref{prop:countkernel}.
We report the columns (i.e., the values of $s$) with small periods in Table \ref{tab:nf4ms}.
The full table is available on QH's personal website\footnote{\url{https://sites.google.com/view/qidonghe/}}.

\subsubsection{Regarding the computation of $N(G,m,s)$ for $G=E_6, E_7, E_8$}\label{subsubsec:ngmse678}

Due to a lack of computational power, we are unable to complete the computation of $N(G,m,s)$ using Corollary \ref{cor:ngms} for the remaining exceptional Lie groups. 
Let us consider $E_{6}$, the smallest of the remaining groups.
Its $P$ matrix is easy to generate, which turns out to have dimensions $441\times 6$.
Problems begin to arise as soon as we try to carry out the next step, namely to generate $\Mcal$.
The naive approach would be to inspect all submatrices of $P$ with $6$ columns, which is out of the question since it is exponentially expensive in the number of rows of $P$.
The next best thing, it seems, is to try to extend Facts \ref{fact:g2fact} and \ref{fact:f4fact} to $E_{6}$ using Algorithm \ref{alg:findp}, which might allow us to cut down the time complexity to polynomial.
Unfortunately, even this is beyond our reach, for even in the best case scenario, it would require us to inspect all $7\times 6$ submatrices of $P$, of which there are $\binom{441}{7}\approx 6\times 10^{14}$ and still too many for us to handle.
Finding an efficient way to generate and manipulate the elements of $\Mcal$ in the spirit of Corollary \ref{cor:ngms} is thus an open problem, which we do not consider in this paper.

\begin{sidewaystable}[]
\begin{tabular}{|c|cc|cc|cc|cc|cc|}
\hline
\multicolumn{11}{|c|}{$1152\cdot N(F_{4},m,s)$} \\ \hline
$s=1$ & $s=2$ & $m\bmod 2$ & $s=3$ & $m\bmod 6$ & $s=4$ & $m\bmod 4$ & $s=6$ & $m\bmod 12$ & $s=8$ & $m\bmod 24$  \\ \hline
\multirow{24}{*}{$1152$} & \multirow{12}{*}{$2304$} & \multirow{12}{*}{$0$} & \multirow{4}{*}{$576m+1152$} & \multirow{4}{*}{$0$} & \multirow{6}{*}{$4608$} & \multirow{6}{*}{$0$} & \multirow{2}{*}{$864m+2304$} & \multirow{2}{*}{$0$} & $864m+1152$ & $0$ \\
 &  &  &  &  &  &  &  &  & $0$ & $1$ \\
 &  &  &  &  &  &  & \multirow{2}{*}{$0$} & \multirow{2}{*}{$1$} & $864m-1728$ & $2$ \\
 &  &  &  &  &  &  &  &  & $0$ & $3$ \\
 &  &  & \multirow{4}{*}{$576m-576$} & \multirow{4}{*}{$1$} &  &  & \multirow{2}{*}{$864m-1728$} & \multirow{2}{*}{$2$} & $864m-3456$ & $4$ \\
 &  &  &  &  &  &  &  &  & $0$ & $5$ \\
 &  &  &  &  & \multirow{6}{*}{$0$} & \multirow{6}{*}{$1$} & \multirow{2}{*}{$0$} & \multirow{2}{*}{$3$} & $864m-5184$ & $6$ \\
 &  &  &  &  &  &  &  &  & $0$ & $7$ \\
 &  &  & \multirow{4}{*}{$576m-1152$} & \multirow{4}{*}{$2$} &  &  & \multirow{2}{*}{$864m-3456$} & \multirow{2}{*}{$4$} & $864m+4608$ & $8$ \\
 &  &  &  &  &  &  &  &  & $0$ & $9$ \\
 &  &  &  &  &  &  & \multirow{2}{*}{$0$} & \multirow{2}{*}{$5$} & $864m-1728$ & $10$ \\
 &  &  &  &  &  &  &  &  & $0$ & $11$ \\
 & \multirow{12}{*}{$0$} & \multirow{12}{*}{$1$} & \multirow{4}{*}{$576m+1728$} & \multirow{4}{*}{$3$} & \multirow{6}{*}{$0$} & \multirow{6}{*}{$2$} & \multirow{2}{*}{$864m+4032$} & \multirow{2}{*}{$6$} & $864m-6912$ & $12$ \\
 &  &  &  &  &  &  &  &  & $0$ & $13$ \\
 &  &  &  &  &  &  & \multirow{2}{*}{$0$} & \multirow{2}{*}{$7$} & $864m-1728$ & $14$ \\
 &  &  &  &  &  &  &  &  & $0$ & $15$ \\
 &  &  & \multirow{4}{*}{$576m-1152$} & \multirow{4}{*}{$4$} &  &  & \multirow{2}{*}{$864m-3456$} & \multirow{2}{*}{$8$} & $864m+4608$ & $16$ \\
 &  &  &  &  &  &  &  &  & $0$ & $17$ \\
 &  &  &  &  & \multirow{6}{*}{$0$} & \multirow{6}{*}{$3$} & \multirow{2}{*}{$0$} & \multirow{2}{*}{$9$} & $864m-5184$ & $18$ \\
 &  &  &  &  &  &  &  &  & $0$ & $19$ \\
 &  &  & \multirow{4}{*}{$576m-576$} & \multirow{4}{*}{$5$} &  &  & \multirow{2}{*}{$864m-1728$} & \multirow{2}{*}{$10$} & $864m-3456$ & $20$ \\
 &  &  &  &  &  &  &  &  & $0$ & $21$ \\
 &  &  &  &  &  &  & \multirow{2}{*}{$0$} & \multirow{2}{*}{$11$} & $864m-1728$ & $22$ \\
 &  &  &  &  &  &  &  &  & $0$ & $23$ \\
                   \hline
\end{tabular}
\caption{A part of the full table for $N(F_{4},m,s)$.}
\label{tab:nf4ms}
\end{sidewaystable}

\appendix
\section{Algorithms}
This section contains three pseudo-codes that summarize our methods. 
\begin{algorithm}
\begin{algorithmic}[1]
\REQUIRE a conjugacy class representative $w_c$
\STATE compute its matrix representation
\STATE compute $w_c\cdot t(\vec{k})=w_ct(\vec{k})w_c^{-1}$
\STATE by equating the $i_j$th entries of $w_c\cdot t(\vec{k})$ and $t(\vec{k})$ for each $j=1,\dots,\ell$, write down a homogeneous system of $\ell$ linear equations in $\ell$ variables
\STATE compute the size of the kernel using Proposition \ref{prop:countkernel}, which will be $\lvert\Fix(w_c)\rvert$
\RETURN $\lvert\Fix(w_c)\rvert$
\end{algorithmic}
\caption{An algorithm for computing the number of fixed points in $E(T,m)$ under the action of a conjugacy class representative $w_c$.\label{alg:fixedpoint}}
\end{algorithm}
\FloatBarrier

\begin{algorithm}
\begin{algorithmic}[1]
\REQUIRE a conjugacy class representative $w_c$ 
\STATE compute its matrix representation 
\STATE compute $w_c\cdot t(\vec{k})=w_ct(\vec{k})w_c^{-1}$
\STATE determine $\sigma_{w_c}$, the signed permutation associated to $w_c$
\STATE find $S_{w_c}$ (Definition \ref{def:sw})
\STATE using Equations \ref{eqn:sfunctiong2f4e8}, \ref{eqn:sfunctione6}, or \ref{eqn:sfunctione7}, find all $S$ such that $S\in\Mcal$, $S\preceq S_{w_c}$, and $s(S)=s$
\STATE for all such $S$'s, use Proposition \ref{prop:countkernel} to determine $g_m(T)$ for $T\preceq S$, $T\in \Mcal$ and use Equation \ref{eqn:g2fgrelation}  to compute $f_m(S)$ 
\STATE add the results for all the $f_m(S)$, which will be $\lvert\Fix_s(w_c)\rvert$
\RETURN $\lvert\Fix_s(w_c)\rvert$
\end{algorithmic}
\caption{An algorithm for computing the number of fixed points in $E(T,m,s)$ under the action of a conjugacy class representative $w_c$.\label{alg:fixedpoints}}
\end{algorithm}

\begin{algorithm}
\begin{algorithmic}[1]
\REQUIRE matrix $P$ of dimensions $n\times l$
\FOR{$p\ge 1$}
\STATE compute the set $H_p$ of distinct RHNFs of the submatrices of $P$ with $l$ columns and no more than $p$ rows 
\STATE compute the set $H_{p+1}$ of distinct RHNFs of the submatrices of $P$ with $l$ columns and no more than $p+1$ rows 
\IF{$H_p=H_{p+1}$ up to adding and removing rows of zeros}
\RETURN $p$
\ELSE
\STATE $p=p+1$
\ENDIF
\ENDFOR
\end{algorithmic}
\caption{A brute-force algorithm for determining the minimum integer $r$ such that every row lattice generated by a submatrix of $P$ with $l$ columns can be generated by a submatrix of $P$ with $l$ columns and no more than $r$ rows.\label{alg:findp}}
\end{algorithm}



\section*{Acknowledgements}

The authors thank the anonymous referees for many helpful comments and Fernando Q. Gouv\^ea for enlightening discussions.


%
%

\bibliographystyle{alphaurl}
\bibliography{CCCpaperCTv5}

\end{document}